\documentclass[preprint,11pt]{article}

\usepackage[dvipdfmx]{graphicx}
\usepackage{epsfig}
\usepackage{graphics}

\usepackage{url}

\usepackage{amssymb}
\usepackage{amsthm}

\usepackage{lineno}

\usepackage{amsmath}
\usepackage{amsthm}
\usepackage{amssymb}
\usepackage{amsfonts}

\usepackage{comment}

\usepackage{cases}

\usepackage{cleveref}
\usepackage{bm}

\newtheorem{theorem}{Theorem}[section]

\newtheorem{lemma}[theorem]{Lemma}
\newtheorem{remark}[theorem]{Remark}

\newcommand{\ud}{\mathrm{d}}

\newcommand{\Dim}{\mbox{dim}\mspace{3mu}}

\title{Explicit Bound for Quadratic Lagrange Interpolation Constant on Triangular Finite Elements}

\author{Xuefeng LIU\footnote{Graduate School of Science and Technology, Niigata University, 8050 Ikarashi 2-no-cho,
Nishi-ku, Niigata City, Niigata 950-2181 Japan. E-mail: xfliu@math.sc.niigata-u.ac.jp (Corresponding author)}, Chun'guang YOU \footnote{LSEC, ICMSEC, Academy of Mathematics and Systems Science, Chinese Academy of
Sciences, Beijing 100190, China. E-mail: youchg@lsec.cc.ac.cn} \thanks{This work is supported by Japan Society for the Promotion of Science, Grand-in-Aid for Young Scientist (B) 26800090 and
Grant-in-Aid for Scientific Research (B) 16H03950.}
}

\begin{document}

\maketitle

\begin{abstract}
For the quadratic Lagrange interpolation function, an algorithm is proposed to provide explicit and
verified bound for the interpolation error constant that appears in the interpolation error estimation.
The upper bound for the interpolation constant is obtained by solving an eigenvalue problem along with
explicit lower bound for its eigenvalues.
The lower bound for interpolation constant can be easily obtained by applying the Rayleigh-Ritz method.
Numerical computation is performed to demonstrate the sharpness of lower and upper bounds of the
interpolation constants over triangles of different shapes. An online computing demo is available at 
\url{http://www.xfliu.org/onlinelab/}.

\end{abstract}



\section{Introduction}

In this paper we aim to provide explicit error estimation for quadratic interpolation operator
$\Pi_2$ defined for functions over triangular elements. Given a triangle $T$, denote the three vertices by
$p_1$, $p_2$, $p_3$, and the mid-points by $p_{12}$, $p_{23}$, $p_{31}$; see Figure \ref{fig:tri}.

\begin{figure}[ht]
\begin{center}
\includegraphics[height=1.3in]{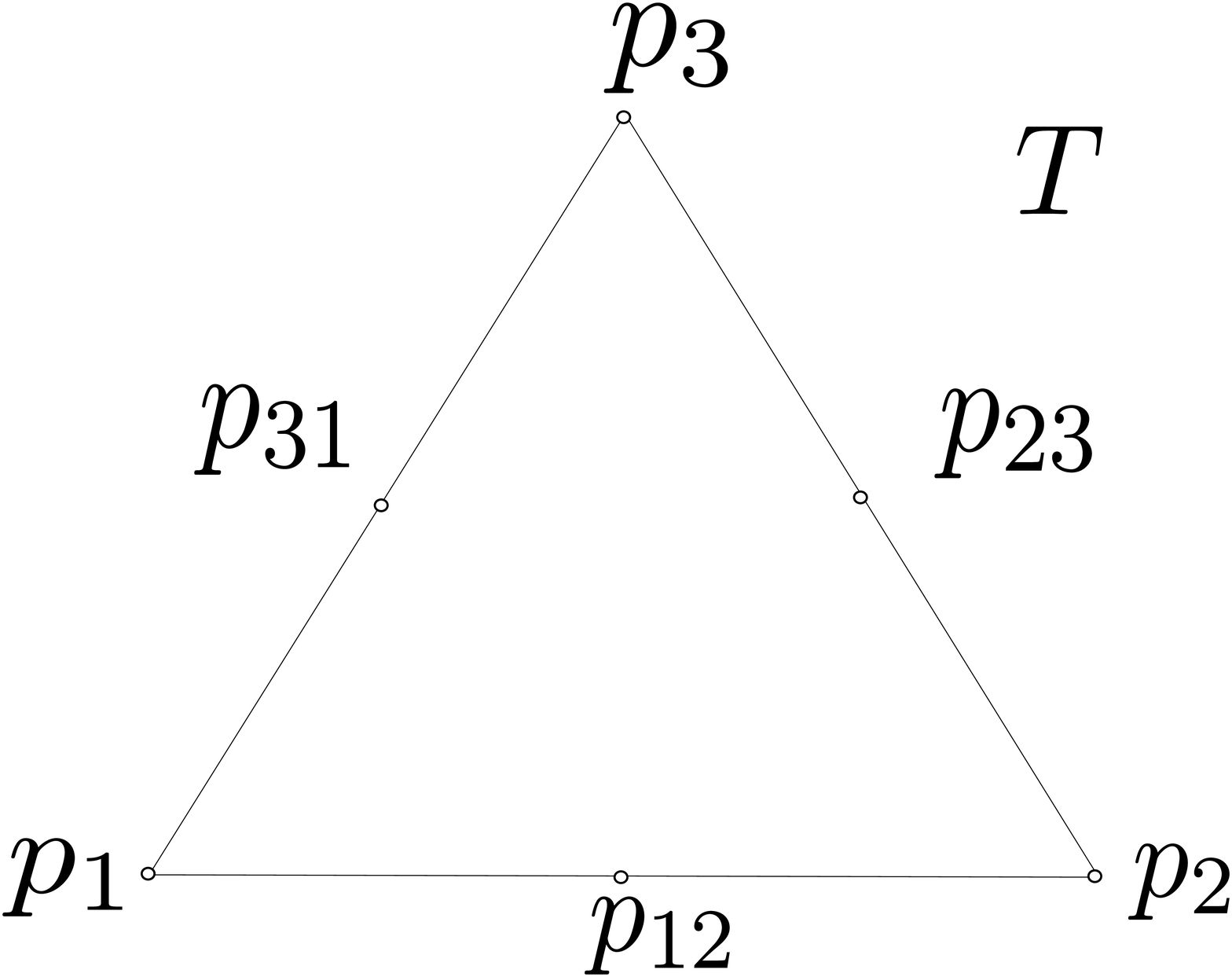}
\caption{\label{fig:tri} Configuration of a triangle element $T$}
\end{center}
\end{figure}

The finite element method (FEM) defined on a triangulation of domains is often used to solve 
partial differential equations, for example, the model problem of Poisson's equation.
In many cases, one can only make sure the $H^2$-regularity of the solution $u$, 
i.e., $u\in H^2(\Omega)$, and the first and second order Lagrange interpolations $\Pi_{1}$ and $\Pi_{2}$ are often used to 
give error estimation for FEM solutions; see, e.g., \cite{Ciarlet2002,Brenner+Scott2008}.

\paragraph{$\Pi_{1}$ interpolation}
Let us first review the definition of $\Pi_{1}$: 
$\Pi_{1}u$ is a linear function that interpolates $u\in H^{2}(T)$  at vertices of $T$, i.e.,
\begin{equation}
\label{eq:def-pi1}
(\Pi_1 u)(p_i) - u (p_i)=0, \quad i =1,2,3\:.
\end{equation}
One of the interpolation error estimation for $\Pi_{1}$ is given by
\begin{equation}
\label{eq:error-est-pi1}
|u - \Pi_1 u|_{1,T} \le \widetilde{C}_T |u|_{2,T}, \,\,\forall\, u\in H^2(T)\:.
\end{equation}
The definition of  $|\cdot|_{i}$ ($i=0,1,2,\cdots$) is taken from the notation of
Sobolev spaces; see \S\ref{preliminaries}.
The explicit bound of constant $\widetilde{C}_{T}$ dates back to the work of 
Natterer \cite{Natterer1975} and Lehmann \cite{R.Lehmann1986}, while recent work can be found in 
Kikuchi-Liu \cite{liu-kikuchi-2010,Kikuchi+Liu2007} and Kobayashi \cite{Kobayashi2011}, etc.
A sharp bound for the interpolation error constant $\widetilde{C}_{T}$ on concrete elements is given as follows (see, e.g., \cite{kobyashi-2015}),
\begin{equation}
\left\{
\begin{array}{ll}
\widetilde{C}_T \le 0.4889 & \!\!\!\!\mbox{ for unit isosceles right triangle; } \\
\widetilde{C}_T \le 0.3186 & \!\!\!\!\mbox{ for unit regular triangle. }
\end{array}
\right.
\end{equation}

\vskip 0.2cm

\paragraph{$\Pi_{2}$ interpolation}
 The second order Lagrange interpolation $\Pi_2$ for $u\in H^2(T)$ is a quadratic polynomial satisfying
\begin{equation}
\label{eq:def-pi2}
(\Pi_2 u)(p_i) - u (p_i)=0, \quad i \in I_0,
\end{equation}
where $I_0$ is a set of indices given by
$I_0:=\{1,2,3,\{12\},\{23\},\{31\}\}$. 
The error estimation for $\Pi_{2}u$ is given as follows,
\begin{equation}
\label{eq:error-est-pi2}
|u - \Pi_2 u|_{1,T} \le C_T |u|_{2,T}, \,\,\forall\, u\in H^2(T)\:.
\end{equation}
Here, the interpolation constant $C_T$ only depends on the shape of triangle $T$ itself.
For $\Pi_{2}$ and even general $k$th order Lagrange interpolation, 
there have been various literatures to investigate the dependence of interpolation error on triangle shapes;
see, e.g., the early work of Jamet \cite{jamet1976estimations} and recent work of Kobayashi-Tsuchiya \cite{kobayashi2015priori}.

To the author's best knowledge, there is no result reported for the upper bound of the constant $C_T$.
In this paper, we will propose a method that gives explicit bound for $C_T$ by
solving eigenvalue problem of corresponding differential operators.
As shown in \Cref{table:constant} of \S\ref{sec:numerical-results}, we have, for example, 
\begin{equation}
\left\{
\begin{array}{ll}
C_T \le 0.2598  & \!\!\!\!\mbox{ for unit isosceles right triangle; } \\
C_T \le 0.1777  & \!\!\!\!\mbox{ for unit regular triangle. }
\end{array}
\right.    
\end{equation}

Compared with the bound for $\widetilde{C}_{T}$, one can see 
that $\Pi_{2}$ has more accurate error estimation than the linear interpolation $\Pi_{1}$. \\

The method proposed in this paper can also be applied to interpolation for $u$ with higher regularity.
For example, for $u \in H^{3}(T)$, the following error estimation holds,
\begin{equation}
\label{eq:error-est-pi2-known}
|u - \Pi_2 u|_{i,T} \le C_{T,i}  |u|_{3,T}, \,\,\forall\, u\in H^3(T) \quad (i=0,1,2)\:.
\end{equation}
The sharp and explicit values of the interpolation constants $C_{T,i}$ are rarely reported. 
A brief discussion for bounding such constants can be found in \S5. \\

\vskip 0.2cm

The structure of the rest of the paper is as follows. In \S\ref{preliminaries}, the characterization of constant $C_T$ through minimization problem and
eigenvalue problems is given, where we can see that to give a lower bound of $C_T$ is quite easy.
In \S\ref{sec:lower-bound}, the upper bound of $C_T$ is discussed through the lower bound of certain eigenvalue, which is main contribution of our paper.
In \S\ref{sec:numerical-results}, the computed lower and upper bound for $C_T$ are given for $T$ with various shapes.
Finally in \S\ref{sec:summary}, we provide a rough idea to bound the constants in \cref{eq:error-est-pi2-known} and point out the future work.

\section{Preliminaries}\label{preliminaries}
In this paper, the domain $\Omega$ of functions is selected as the triangle $T$.
The standard notation is used for Sobolev function spaces $W^{k,p}(\Omega)$.
The associated norms and semi-norms are denoted by $\|\cdot\|_{k,p,\Omega}$, $|\cdot|_{k,p,\Omega}$, respectively
(see, e.g., Chapter 1 of \cite{Brenner+Scott2008} and Chapter 1 of \cite{Ciarlet2002}).
Particularly, for special $k$ and $p$, we use abbreviated notation as $H^k(\Omega)=W^{k,2}(\Omega)$, $|\cdot|_{k,\Omega} = |\cdot|_{k,2,\Omega}$,
$L^p(\Omega)=W^{0,p}(\Omega)$.
The set of polynomials over $T$ of up to degree $k$ is denoted by $P_k(T)$.

The following subspace $V_0$ of $H^2(T)$ will play an important role in bounding the interpolation constant $C_T$.
\begin{equation}
\label{eq:def-V-initial}
V_0:= \{ v \in H^2(T) \:|\ v(p_i)=0, i=1,2,3 \}\:.
\end{equation}
The intersection of $V_0$ and $P_{k}(T)$ is denoted by
\begin{equation}
\label{def:polynomial-conforming-space}
P_{k,0}(T): = V_0 \cap P_{k}(T)=\{ v \in P_k(T) ~|~ v(p_i) =0, i=1,2,3 \}\:.
\end{equation}
Define bilinear forms $M$, $N$ as follows,
\begin{equation}
\label{eq:def-M-N}
M(u,v) := \int_{T} D^2 u \cdot D^2 v \:\: \mbox{d}T, \quad
N(u,v) := \int_{T} \nabla (u - \Pi_2 u)  \cdot \nabla (v - \Pi_2 v) \:\: \mbox{d}T\:.
\end{equation}
Here, $D^2$ is the second order derivative of $u$, i.e.,
$$
D^2u:=(\frac{\partial^2 u}{\partial x^2}, \frac{\partial^2 u}{\partial x \partial y}, \frac{\partial^2 u}{\partial y \partial x},  \frac{\partial^2u}{\partial y^2})\:.
$$
It is easy to see that $M$ is an inner product of $V_0$ and $N$ is a positive semi-definite bilinear form of $V_0$.

Define Rayleigh quotient $R$ for $u  \in H^2(T)$,
\begin{equation}
\label{eq:rayleigh-quotient}
R(u):=\frac{|u|^2_{2,T}}{|u-\Pi_2u|^2_{1,T}} = \frac{M(u,u)}{N(u,u)} \:.
\end{equation}
Let $\Pi_{1} u $ be the linear Lagrange interpolation of $u \in H^2(T)$. 
Since $(u-\Pi_1 u)(p_i)=0, (i=1,2,3)$, we have $(u-\Pi_1 u) \in V_{0}$.
Noticing that $R(u)=R(u-\Pi_1u)$, the optimal constant $C_T$ in \cref{eq:error-est-pi2} can be defined through the quantity $\overline{\lambda}$,
\begin{equation}
\label{eq:minimizer}
C_T^{-2}  := \overline{\lambda} = \inf_{ u \in H^2(T), |u-\Pi_2u|_{1,T}\not=0} R(u)
= \inf_{ u \in V_0, |u-\Pi_2u|_{1,T}\not=0} R(u) \:.
\end{equation}
The existence of minimizer of \cref{eq:minimizer} can be proved by the standard compactness argument.
Thus the infimum above is actually a minimum, which corresponds to the smallest eigenvalue of certain eigenvalue problem.

\paragraph{Upper bound for $\overline{\lambda}$}
By using the Rayleigh-Ritz method, it is easy to
give an upper bound for $\overline{\lambda}$, i.e., the lower bound for the interpolation constant.
The Rayleigh-Ritz bound $\overline{\lambda}_{upper}$ along with the usage of $P_{k,0}(T)$ is given by
$$
\overline{\lambda}_{upper} := \min_{ u \in P_{k,0}(T), |u-\Pi_2u|_{1,T}\not=0 } R(u)\:.
$$
Quantity $\overline{\lambda}_{upper}$ can be calculated by solving a matrix eigenvalue problem.
Take a basis of
$P_{k,0}(T)$ as $\{ \phi_i \}_{i=1}^n$, $n=\mbox{dim}(P_{k,0}(T))$. Define matrix $A$ and $B$ by
$$
A=\Big( M(\phi_i,\phi_j)\Big)_{i,j=1,\cdots, n}, \quad
B=\Big( N(\phi_i,\phi_j)\Big)_{i,j=1,\cdots, n}.
$$
Notice that $A$ is positive definite and $B$ is positive semi-definite.
The value of $\overline{\lambda}_{upper}^{\:\:-1}$ is given by the maximum eigenvalue of the problem
$$
Bx=\mu Ax\:.
$$
The calculated upper bounds of the $\overline{\lambda}_{upper}$, i.e., the lower bounds of interpolation constant, are shown in \Cref{sec:numerical-results},
where the poloynomial functions in $P_{6,0}(T)$ are utilized.

\paragraph{Preparation for lower bound of $\overline{\lambda}$}

Below, we show the important property of the minimizer of (\ref{eq:minimizer}).
\begin{lemma} \label{lem:minimizer} The minimizer $u_0$ of (\ref{eq:minimizer}) is orthogonal to $P_2(T)$ with respect to $M(\cdot, \cdot)$.
\end{lemma}
\begin{proof}
Suppose $u_0$ minimizes $R(u)$ for all $u$ satisfying $|u-\Pi_2u|_{1,T}\not=0$.
For any $q\in P_2(T)$, define $g(t):= R(u_0+tq)$, then ${g'(t)}|_{t=0}=0$. That is,
$$
M(u_{0},q)=\overline{\lambda} N(u_{0}, q) 
$$
Noticing that $q-\Pi_{2}q=0$ and hence $N(u_0,q)= 0$, we have  $M(u_0,q)=0$.
\end{proof}

Let us introduce the complement space of $P_{2,0}(T)$ in $V_0$, with respect to the inner product $M(\cdot, \cdot)$.
\begin{equation}
\label{eq:def-W}
W:=\{ v \in V_0 ~|~ M(v,p)=0, ~~\forall p \in P_{2,0}(T) \}\:.
\end{equation}
In space $W$, $N$ becomes a positive definite bilinear form and thus can be regarded as an inner product of $W$.
From \Cref{lem:minimizer}, it is easy to see that the quantity $\overline{\lambda}$ can be characterized by
\begin{equation}
\label{eq:minimizer_with_W}
\overline{\lambda} = \inf_{ u \in W, u \not=0} R(u) \:.
\end{equation}

%

\paragraph{Eigenvalue problem for $\overline{\lambda}$}
The quantity $\overline{\lambda}$ is in fact the minimal eigenvalue of the following eigenvalue problem:
Find $(\lambda, u)  \in \mathbb{R} \times W$ such that
\begin{equation}
\label{eq:eig-for-constant}
M(u,v)=\lambda N(u,v), \quad \forall v \in {W}\:.
\end{equation}
Such an eigenvalue problem has the eigenvalues distributed as $0 < \lambda_1 \le \lambda_2 \le \cdots$.
The quantity $\overline{\lambda}$ is just given by $\lambda_1$.

\section{Lower bound for the eigenvalue $\overline{\lambda}$}
\label{sec:lower-bound}

The eigenvalue problem in (\ref{eq:eig-for-constant}) can be solved approximately by using finite element method (FEM).
Generally, it is difficult to give explicit error estimation for the computed approximation of eigenvalues.
Recently, Liu \cite{Liu-2015} proposed a framework to calculate explicit eigenvalue bounds for compact self-adjoint differential operators.
Such a framework can be implemented with the usage of finite element method.

In this section, we first introduce the finite element method to be used for approximate eigenvalue computation.
Then we apply the framework of  \cite{Liu-2015} to bound the eigenvalue corresponding to the interpolation constant.

\subsection{Finite element approximation for eigenvalues}
%
Take $T$ as the domain and $\mathcal{T}^h$ as a triangular subdivision of $T$.
For each element $K$ of $\mathcal{T}^h$, denote by $h_K$ the longest edge length of $K$ and
the mesh size $h$ is defined by
\begin{equation}
\label{def:mesh-size}
h:=\max_{K \in \mathcal{T}^h} h_K \:.
\end{equation}
Particularly, we require that the triangulation $\mathcal{T}^h$ for $\Omega(=T)$ has $3$ mesh nodes that
coincide with three mid-points $p_{12}$, $p_{23}$ and $p_{31}$ (see \Cref{fig:tri_mesh_special}), which can be easily done by
performing bisections for the original triangle $T$. With such kind of mesh,
the interpolation error estimation needed in bounding the eigenvalue will be much easier.

\begin{figure}[ht]
\begin{center}
\includegraphics[width=1.8in]{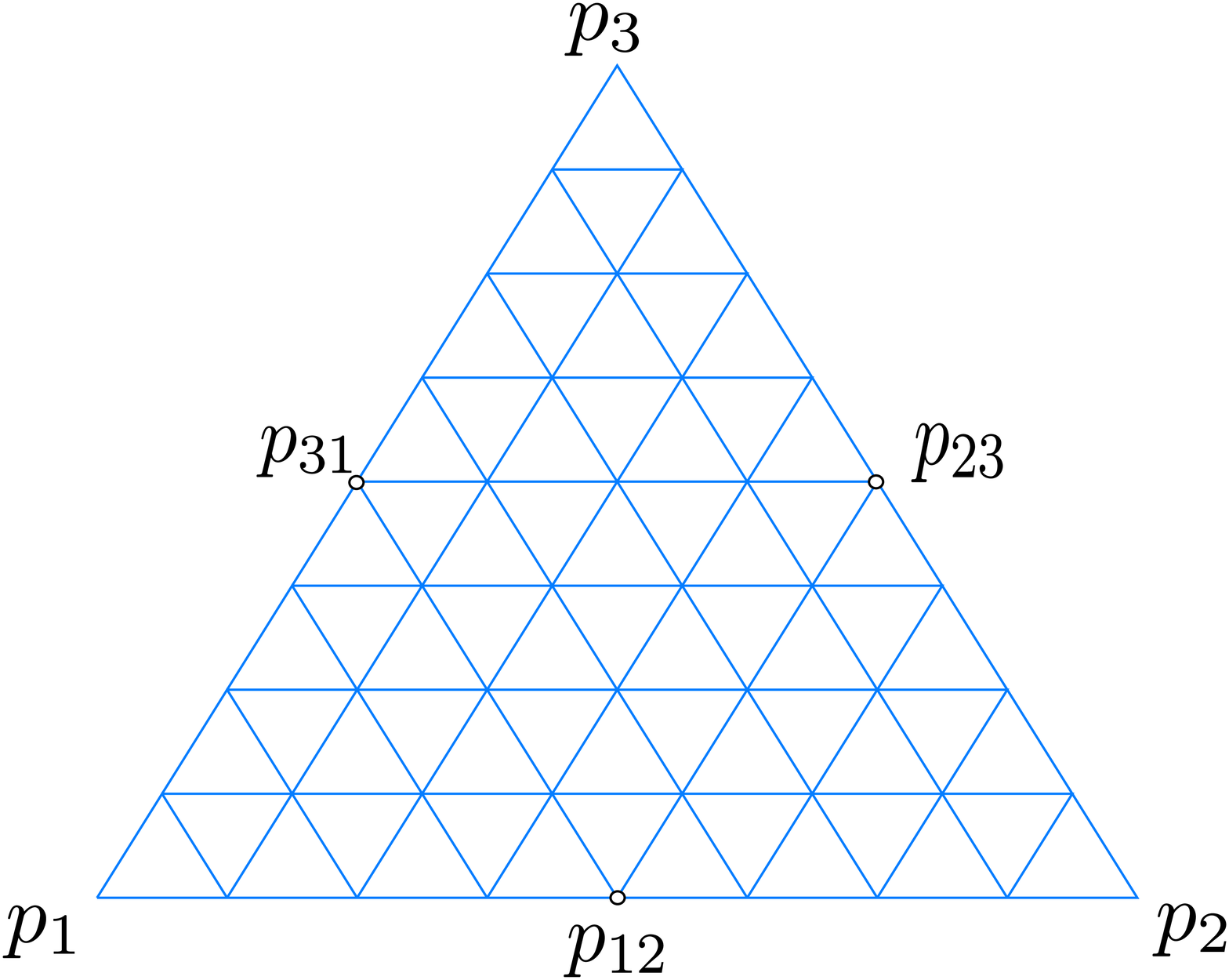}
\caption{\label{fig:tri_mesh_special} Special triangulation needed for lower eigenvalue bound computation}
\end{center}
\end{figure}

Define the Fujino-Morley FEM space $V_{FM}^h$ over $\mathcal{T}^h$ \cite{morley1968triangular,Fujino1971}.
\begin{equation}\label{fujino-morley-fem}
\begin{aligned}
V_{FM}^h & := \left\{v_h \: | \right.
     v_h|_K\in P_2(K) \mbox{ for each } K\in \mathcal{T}^h;  v_h \mbox{ is continuous at the nodes;} \\
    & \int_{e}  \left. \frac{\partial v_h}{\partial \vec{n}} \right|_{K_1} -
    \left. \frac{\partial v_h }{\partial \vec{n}} \right|_{K_2} \mbox{d}s = 0
    \left. \mbox{ for } e = K_1 \cap K_2; v_h(p_i) =0, i=1,2,3\right\}\:. 
\end{aligned}
\end{equation}
Here, since $V_{FM}^h \not\subset H^1(T)$, we introduce the discrete gradient operator $\nabla_h$, which is defined element-wisely.
Given $u_h \in V_{FM}^h$, $\nabla_h u_h \in (L^2(T))^2$ and each part of $\nabla_h u_h$ is piecewise linear polynomial, i.e.,
$(\nabla_h u_h)|_K = \nabla (u_h|_K) $. For $u\in H^1(T)$, we have $\nabla u = \nabla_h u$.

Extend the bilinear forms $M$ and $N$ to $V^h_{FM}$: for any $u_{h}, v_{h} \in V^{h}_{FM}$,
\begin{equation}\label{fujino-morley-Mh-Nh}
\begin{split}
M_h(u_h,v_h):=  & \sum_{K\in \mathcal{T}^h} \int_{K} D^2 u_h \cdot D^2 v_h \mbox{d}K,\quad  \\
N_h(u_h,v_h):=  & \sum_{K\in \mathcal{T}^h} \int_{K} \nabla (u_h - \Pi_2 u_h) \cdot \nabla (v_h - \Pi_2 v_h) \mbox{d}K\:.
\end{split}
\end{equation}
The extension of $R$ for $V^h_{FM}$, which is also denoted by $R$, is defined by
$$
R(u_h):= \frac{ M_h(u_h,u_h) }{ N_h(u_h,u_h) }\quad \mbox{for any } u_h \in V_{FM}^h, N_h(u_h,u_h) \not= 0 \:.
$$

Define the minimization problem of $R$ over $V_{FM}^h$, for which the minimizer is denoted by
$u_{h,0}$.
\begin{equation}
\label{eq:def-lambda_h_1}
\overline{\lambda}_{h} := \min_{ v_{h} \in V_{FM}^h, |v_{h}-\Pi_2v_{h}|_{1,T}\not=0 } R(v_{h}) = R(u_{h,0}).
\end{equation}

\begin{lemma} \label{lem:minimizer2} The minimizer $u_{h,0}$ is orthogonal to $P_2(T)$ with respect to $M_h(\cdot, \cdot)$.
\end{lemma}
\begin{proof}
The proof is omitted since it is analogous to the one for \Cref{lem:minimizer}.
\end{proof}

As an approximation to space $W$, let us introduce the complement space, denoted by $W^h$, of $P_{2,0}(T)$ in $V_{FM}^h$ with respect to inner product $M_h$.
\begin{equation}
\label{eq:def-W-h}
W^h:=\{ v \in V_{FM}^h ~|~ M_h(v,p)=0, ~~\forall p \in P_{2,0}(T) \}\:.
\end{equation}
The bilinear forms $M_h$ and $N_h$ are both positive definite in $W^h$.
For function $u \in W+W^h$, introduce the norm $\|\cdot\|_M$ and $\|\cdot\|_N$ by
\begin{equation}
\label{eq:def-M-N-norm}
\|u\|_M := \sqrt{M_h(u,u)}, \quad  \|u\|_N := \sqrt{N_h(u,u)}\:.
\end{equation}

By using $W^h$, the quantity $\overline{\lambda}_{h}$ can be characterized by
\begin{equation}
\label{eq:def-lambda_h_2}
\overline{\lambda}_{h} := \min_{ u \in W^h, u\not=0 } R(u)\:.
\end{equation}

\paragraph{Fujino-Morley interpolation operator $\Pi^h_{FM}$}
The Fujino-Morley interpolation $\Pi^h_{FM}$ is defined element-wisely.
For an element $K$, denote the edges by $e_1$, $e_2$, $e_3$ and the vertices by $v_1$, $v_2$, $v_3$.
Given $u \in H^2(T)$,  $(\Pi^h_{FM} u) |_K$ is a piece-wise quadratic polynomial such that
\begin{equation}
\int_{e_i} \nabla (u- (\Pi^h_{FM} u)|_K)\cdot \vec{n} \ud s=0, ~
\quad u(v_i) - (\Pi^h_{FM}(u)) (v_i) =0 \quad (i=1,2,3)\:.
\end{equation}
It is easy to verify that $\Pi^h_{FM} u \in V_{FM}^h$  for $u \in  V_0$. \\

\paragraph{Orthogonality of $\Pi^h_{FM}$}
Given a triangle element $K \in \mathcal{T}^h$, for $u \in H^2(K)$, by integration by part, it is easy to see that
$$
\int_K D^2 (\: \Pi^h_{FM}u|_K - u|_K) \cdot  D^2 v  \:\mbox{d}{K} =0, \quad \forall v \in P_2(K)\:.
$$
Thus, the following orthogonality holds.
\begin{equation}
\label{eq:interp_is_project}
M_h(u-\Pi_{FM}^hu, v_{h})=0 , \quad \forall v_{h} \in V_{FM}^h\:.
\end{equation}
From the above orthogonality of $\Pi^h_{FM}$,
it is easy to see that $\Pi_{FM}^h u \in W^h$ for $u \in W$. Therefore, the interpolation operator $\Pi_{FM}^h$ is  just the projection operator that maps  $W$ to $W^h$ with respect to $M_h(\cdot, \cdot)$.\\

Due to the special mesh as shown in \Cref{fig:tri_mesh_special},
$(u - \Pi^h_{FM}u)$ vanishes on  $6$ points $\{p_i\}_{i \in I_0}$. Therefore, $\Pi_2 (u - \Pi^h_{FM}u ) =0$ and
$$
(u-\Pi^h_{FM} u )-\Pi_2 (u-\Pi^h_{FM} u ) = (u-\Pi^h_{FM} u )\:.
$$
Thus the $\|\cdot\|_N$ norm of $(u-\Pi^h_{FM} u)$ can be given as follows.
\begin{equation}
\label{eq:norm-equivalence}
\| u-\Pi^h_{FM} u \|_{N} =  \| \nabla_h ( u-\Pi^h_{FM} u )\|_{0,\Omega}\:.
\end{equation}

The following lemma provides a constant that will be used for the interpolation error estimation.
\begin{lemma}[\cite{Liu-2015}]
\label{lem:ch_value} Given a triangle $K$, denote the edges by $\{ e_i\}_{i=1}^3$ and let $h_K$ be the largest edge length.
For $v\in H^1(K)$ such that $\int_{e_i}v\mbox{d}s=0$, $i=1,2,3$, we have
$$
\|v\|_{0,K} \le  0.1893\:  h_K \:  |v|_{1,K}\:.
$$
\end{lemma}

By using the constant in Lemma \ref{lem:ch_value}, we have the following theorem.
\begin{theorem}
Let $h$ be the mesh size of $\mathcal{T}_h$ (see \cref{def:mesh-size}).
Then the following estimation for $\Pi^h_{FM} $ holds.
\begin{equation}\label{thm:interpolation-estimation}
\|u-\Pi^h_{FM} u\|_{N} \le 0.1893 h \|u - \Pi^h_{FM}u \|_{M}, \,\,\forall\, u \in H^2(T)\:.
\end{equation}
\end{theorem}

\begin{proof}
Given an element $K$  of $\mathcal{T}^h$, the edges of which are denoted by
$\{e_i\}_{i=1}^3$ and the longest edge length being $h_K$.
Take $w := u - \Pi_{FM}^h u$. From the definition of Fujino-Morley interpolation, we have
$$
\int_{e_i} \frac{\partial w}{\partial\vec{n}} \ud s = \int_{e_i} \frac{\partial w}{\partial\vec{\tau}} \ud s = 0 \quad (i=1,2,3)\:,
$$
where $\vec{n}$ and $\vec{\tau}$ are the unit outward normal and unit tangent direction, respectively.
Hence,
$$
\int_{e_i}\frac{\partial w}{\partial x} \ud s = \int_{e_i}\frac{\partial w}{\partial y} \ud s = 0 \quad  (i=1,2,3).
$$
From the estimation in \Cref{lem:ch_value}, we have
$$
\|\frac{\partial w}{\partial x}\|_{0,K} \leq 0.1893 h_K \|\nabla(\frac{\partial w}{\partial x})\|_{0,K},
\quad
\|\frac{\partial w}{\partial y}\|_{0,K} \leq 0.1893 h_K \|\nabla(\frac{\partial w}{\partial y})\|_{0,K}.
$$
By the summation of the square of above integral over all $K$ of $\mathcal{T}^h$ and applying the equation (\ref{eq:norm-equivalence}),
we can easily draw the conclusion.
\end{proof}

\paragraph{Eigenvalue problem in finite element space}
Consider the eigenvalue problem defined over $W^h$: Find $(\lambda_h, u_h)\in \mathbb{R}\times W^h$ such that
\begin{equation}
\label{eq:eig-fem}
M_h(u_h, v_h) = \lambda_h N_h(u_h, v_h), \,\,\forall v_h \in W^h.
\end{equation}
The above eigenvalue problem has eigenvalues as follows,
$$
0< \lambda_{h,1} \le \lambda_{h,2} \cdots \le \lambda_{h,N_0} \quad   ( N_0:=\Dim(W^h))\:.
$$
Then the quantity $\overline{\lambda}_{h}$ (see \cref{eq:def-lambda_h_1} and \cref{eq:def-lambda_h_2}) is given by the smallest eigenvalue of the eigenvalue problem in (\ref{eq:eig-fem}), i.e.,
$\overline{\lambda}_{h} = \lambda_{h,1}$.

\begin{remark}
In practical computing, the space $W^h$ is not needed to be constructed explicitly.
Instead, we solve the eigenvalue problem defined by
\begin{equation*}
\mbox{Find $(\eta_h, u_h)\in \mathbb{R}\times V_{FM}^h$ s.t. }
N_h(u_h, v_h) = \eta_h M_h(u_h, v_h), \,\,\forall v_h \in V_{FM}^h.
\end{equation*}
The above eigenvalue problem has the eigenvalues distributed as follows.
$$
\eta_{h,1} \ge \eta_{h,2} \ge  \cdots \ge \eta_{h,N_0} >  \eta_{h,N_0+1} = \cdots = \eta_{h,N_1} =0\:.
$$
where $N_1 =\mbox{\rm dim}(V_{FM}^h)$ and $N_0 =\mbox{\rm dim}(W^h)$. Moreover, for the positive eigenvalue $\eta_{h,i}$, we have
$$
\lambda_{h,i} =\eta_{h,i}^{-1}  \quad (i=1, \cdots, N_0)\:.
$$
\end{remark}

\subsection{An abstract framework of lower eigenvalue bounds}
\label{subsec:frame}

Let us shape out the framework with the following assumptions.

\begin{enumerate}
\item [(A1)]
$V$ is a real infinite dimensional Hilbert space on $\Omega$ with the inner product $M(\cdot, \cdot)$ and
the corresponding norm $\|\cdot\|_M := \sqrt{M(\cdot,\cdot)}$.

\item [(A2)]
$N(\cdot, \cdot)$ is also an inner product of $V$.
The corresponding norm $\|\cdot\|_N := \sqrt{N(\cdot,\cdot)}$ is compact with respect to $\|\cdot\|_M$, i.e.,
every sequence in $V$ which is bounded in $\|\cdot\|_M$ has a subsequence which is Cauchy in $\|\cdot\|_N$.

\item [(A3)]
$V^h$ is a finite dimensional space of real functions over $\Omega$. Define $V(h) := V + V^h = \{v+v_h\,|\, v\in V, v_h\in V^h\}$.

\item[(A4)]
Bilinear forms $M_h(\cdot, \cdot)$ and $N_h(\cdot, \cdot)$ on $V(h)$ are extension of $M(\cdot, \cdot)$ and $N(\cdot, \cdot)$ to $V(h)$ such that
\begin{enumerate}
\item[-] $M_h(u,v) = M(u, v)$, $N_h(u, v) = N(u, v)$ for all $u, v\in V$.
\item[-] $M_h(\cdot, \cdot)$ and $N_h(\cdot, \cdot)$ are symmetric and positive definite on $V(h)$.
\end{enumerate}
\end{enumerate}

It is easy to confirm that $V(h)$ is also a Hilbert space, along with inner product $M_h(\cdot, \cdot)$ and  $N_h(\cdot, \cdot)$.
The norms corresponding to $M_h$ and $N_h$ are still denoted by $\|\cdot\|_M$ and $\| \cdot \|_N$, respectively.\\

In \cite{Liu-2015}, the following two eigenvalue problems are considered.

(EVP) Find $(\lambda,u)\in \mathbb{R}\times V$ such that
\begin{equation}\label{eigen-problem}
M(u, v) = \lambda N(u, v), \quad \forall v \in V.
\end{equation}

(EVP-h) Find $(\lambda_h,u_h)\in \mathbb{R}\times V^h$ such that
\begin{equation}\label{eigen-problem-fem}
M_h(u_h,v_h)=\lambda_h N_h(u_h,v_h), \quad \forall\, v_h\in V^h.
\end{equation}

Suppose $\Dim(V^h)=n$.
The eigenvalues of (EVP) and (EVP-h), in an increasing order, are denoted by
$\{\lambda_k\}_{k=1}^\infty$ and  $\{\lambda_{h,k}\}_{k=1}^n$, respectively.

\begin{theorem}\label{liu-amc-main}
Let $P_h:V(h)\mapsto V^h$ be the projection with respect to inner product $M_h$, i.e., for any $u\in V(h)$
\begin{equation}\label{M-projection}
M_h(u-P_h u, v_h) = 0, \quad \forall v_h\in V^h.
\end{equation}
Suppose there exists a quantity $C_h$ such that
\begin{equation}\label{M-projection-estimate}
\|u - P_h u\|_N \leq C_h \|u - P_h u\|_M,    \quad \forall u\in V\:.
\end{equation}
Let $\lambda_k$ and $\lambda_{h,k}$ be the ones defined in \cref{eigen-problem} and \cref{eigen-problem-fem}, then we have
\begin{equation}\label{framework lower bounds}
\lambda_k \ge \frac{\lambda_{h,k}}{1+\lambda_{h,k}C_h^2} \quad (k=1,2,\cdots,n).
\end{equation}
\end{theorem}

\begin{remark}
The above theorem does not require $V^h\subset V$. Thus, we can use non-conforming finite element methods to obtain lower eigenvalue bounds.
\end{remark}

\subsection{ Calculation of lower bound of $\overline{\lambda}$}

To apply the  framework introduced in \Cref{subsec:frame}, we take the following settings.
\begin{enumerate}
\item [(S1)] $\Omega:=T$, $V:=W$, $V^h:=W^h$;
\item [(S2)] $M,N$: the ones defined in \cref{eq:def-M-N};
\item [(S3)] $M_h, N_h$: the ones defined in \cref{fujino-morley-Mh-Nh};
\item [(S4)] $P_h:=\Pi_{FM}^h$: from \cref{eq:interp_is_project} and \cref{thm:interpolation-estimation}, we have
$$
\| u - P_h u  \|_N  \le 0.1893h \| u - P_h u  \|_M\,, ~~\forall u \in V.
$$
\end{enumerate}

By applying the result in \Cref{liu-amc-main}, we can give explicit lower bound for eigenvalues in (\ref{eq:eig-for-constant}).
\begin{theorem}
\label{thm:final_conclusion}
For the eigenvalues defined in  (\ref{eq:eig-for-constant}) and (\ref{eq:eig-fem}), we have
\begin{equation}\label{fm-lower-bound}
\lambda_k \geq \frac{\lambda_{h,k}}{1+\lambda_{h,k} (0.1893h)^2} \quad (k=1, \cdots, N_0)\:.
\end{equation}
\end{theorem}

\begin{remark}
In this paper, we are only concerning the lower bound of the first eigenvalue of  \cref{eq:eig-for-constant}.
In fact, with lower bound for the second eigenvalue provided in the above theorem, we can further apply Lehmann-Goerisch's method
to have a sharp bound for the first eigenvalue. Refer to \cite{liu-2013-2} for the case of error constant estimation
for the constant interpolation functions. The high-precision bound for the interpolation constant $C_T$ will be considered in future work.
\end{remark}

\section{Numerical computations}
\label{sec:numerical-results}

\noindent
In this section, we apply the Fujino-Morley finite element along with the bound
in Theorem \ref{thm:final_conclusion} to calculate the upper bound for the interpolation constant $C_T$.
To demonstrate the precision of upper bound, the lower bounds are also calculated by using Rayleigh-Ritz method along with
the usage of polynomial space $P_{k,0}(T)$.
The two-side bounds for constant $C_T$ are given for several concrete triangles, for example, the unit isosceles right triangle, the regular triangle.

In order to estimate the rounding error in floating-point number computing, the interval arithmetic is applied
in the numerical computation. As a consequence, the results can be expected to be
mathematically correct. The method of Behnke \cite{Behnke-1991} (along with the INTLAB toolbox, developed by Rump \cite{Rump-1999})
is used to give the verified eigenvalue bounds for the generalized matrix eigenvalue problems.

Suppose the two vertices of triangle domain are fixed to be $(0,0)$ and $(1,0)$, and the third vertex is
given by $(a,b)$ with variable values.
In Table \ref{table:constant},
for different choices of $(a,b)$, we display the verified lower and upper bounds for $\lambda_1$ and
$C_T$ , which are denoted by $\lambda_{\text{low}}$, $\lambda_{\text{upper}}$, $C_{\text{low}}$ and $C_{\text{upper}}$, respectively.
The mesh size for triangulation is taken as $h=1/64$.
The upper bound for $\lambda_1$ is obtained by using the Rayleigh-Ritz bound with polynomial function space $P_{6,0}(T)$.
In \Cref{fig:std-eigfunc} and \Cref{fig:regular-eigfunc}, we display the 3D graph and color-map for the eigenfunction
corresponding to $\lambda_1$.

\begin{figure}[ht]
\begin{center}
\includegraphics[height=5cm,angle=0]{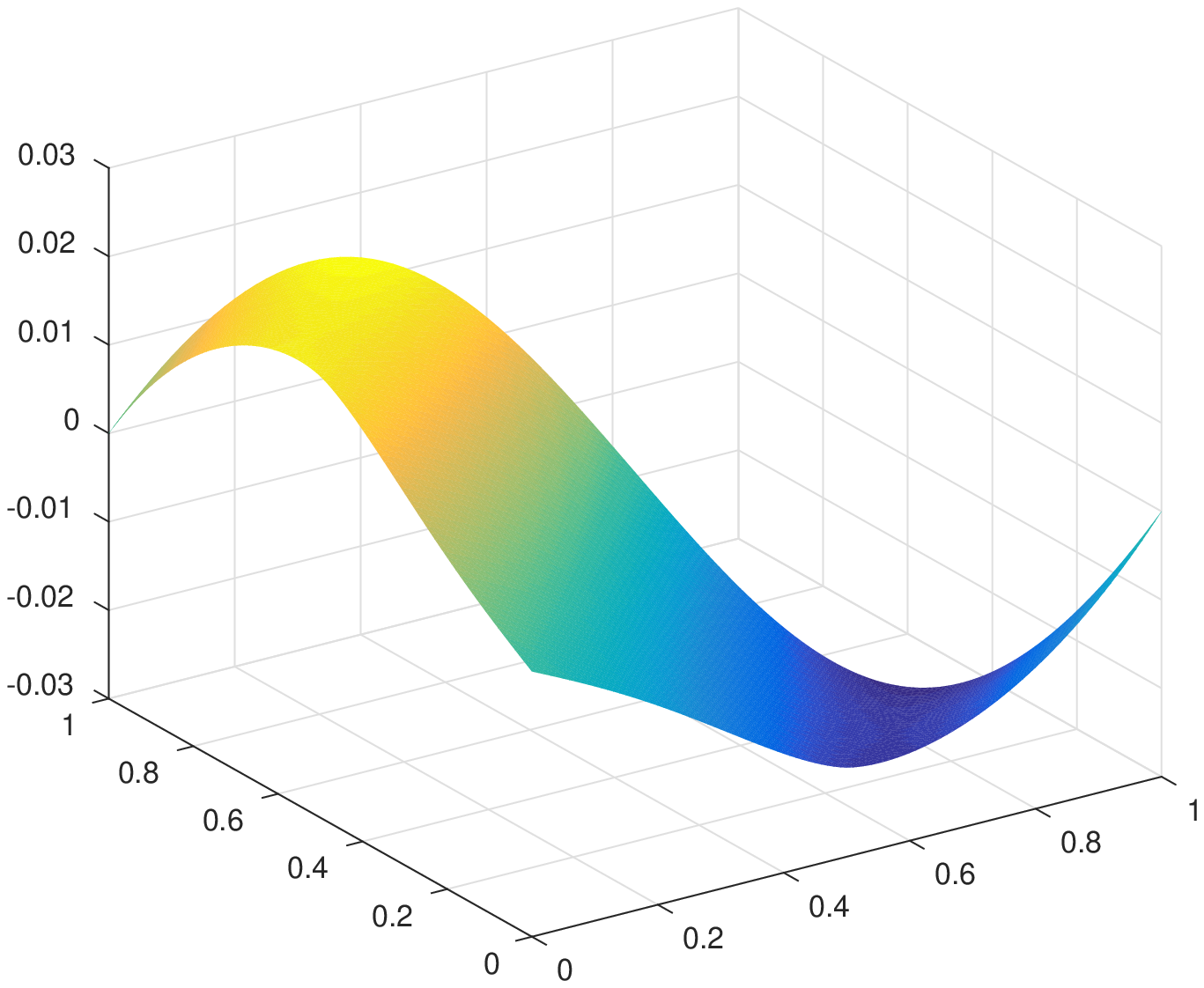}~
\includegraphics[height=5cm,width=5.2cm, angle=0]{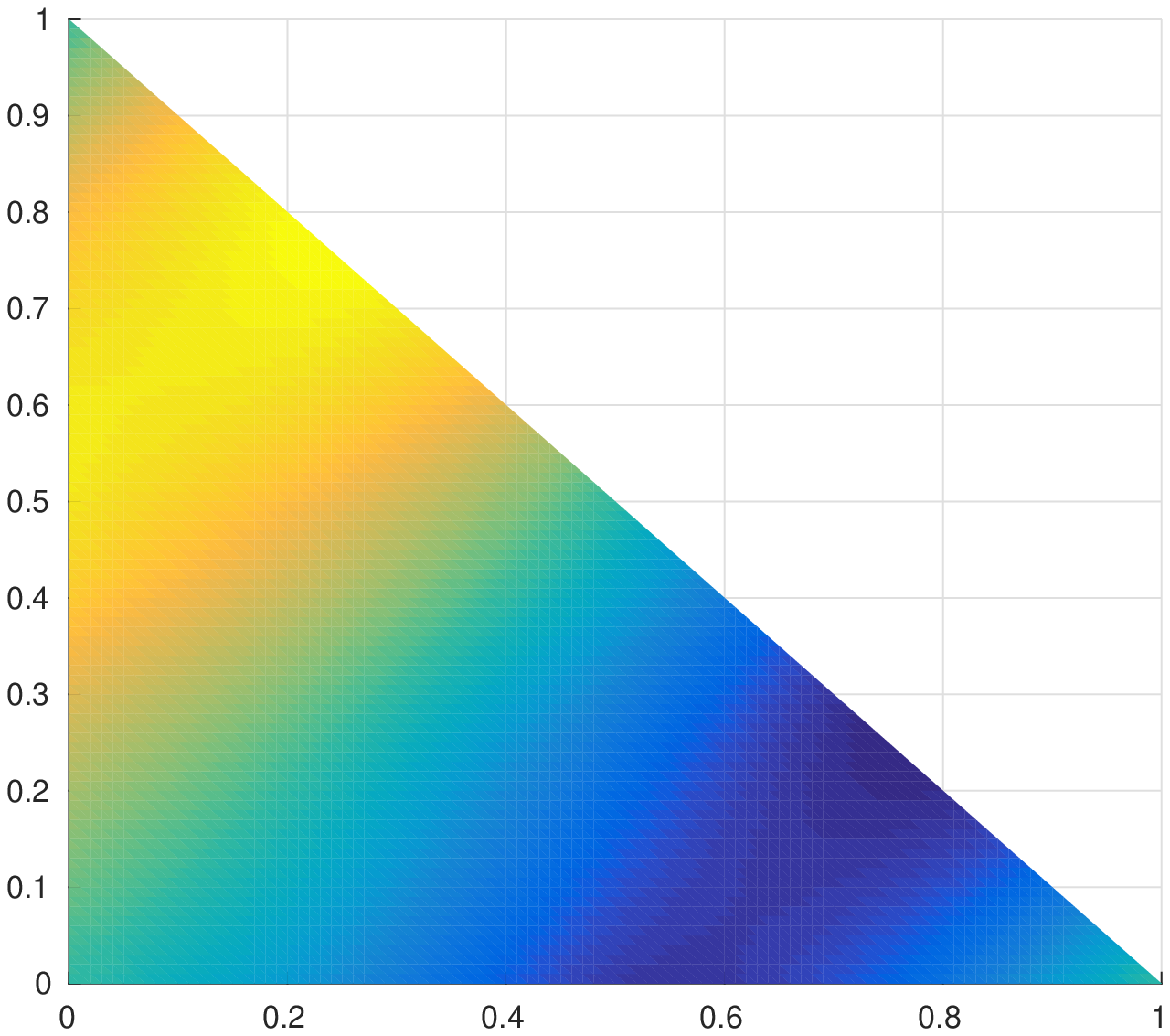}
\end{center}
\caption{\label{fig:std-eigfunc} The first eigenfunction on the unit isosceles right triangle}
\end{figure}

\begin{figure}[ht]
\begin{center}
\includegraphics[height=5cm,angle=0]{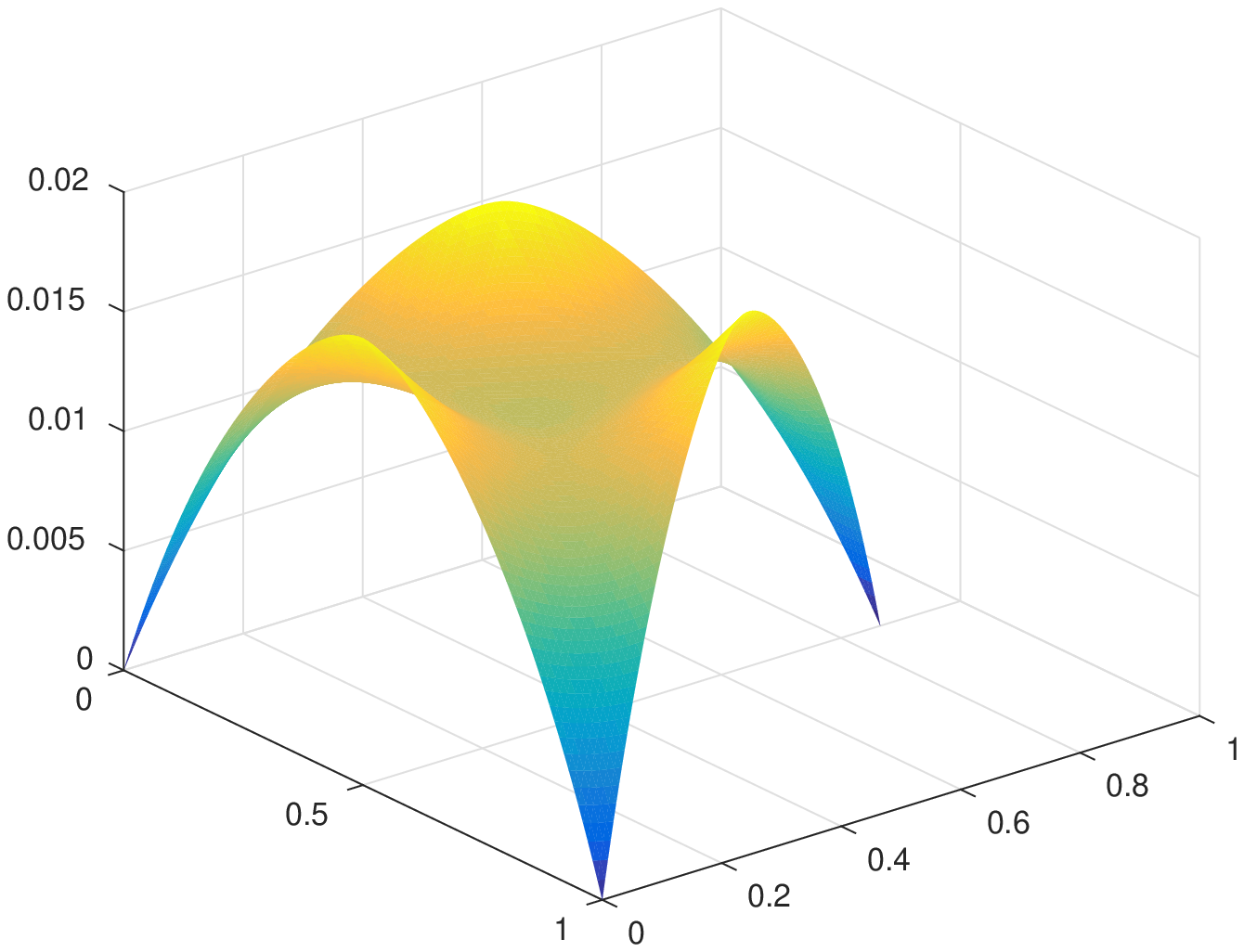}~
\includegraphics[height=5cm,width=5.4cm, ,angle=0]{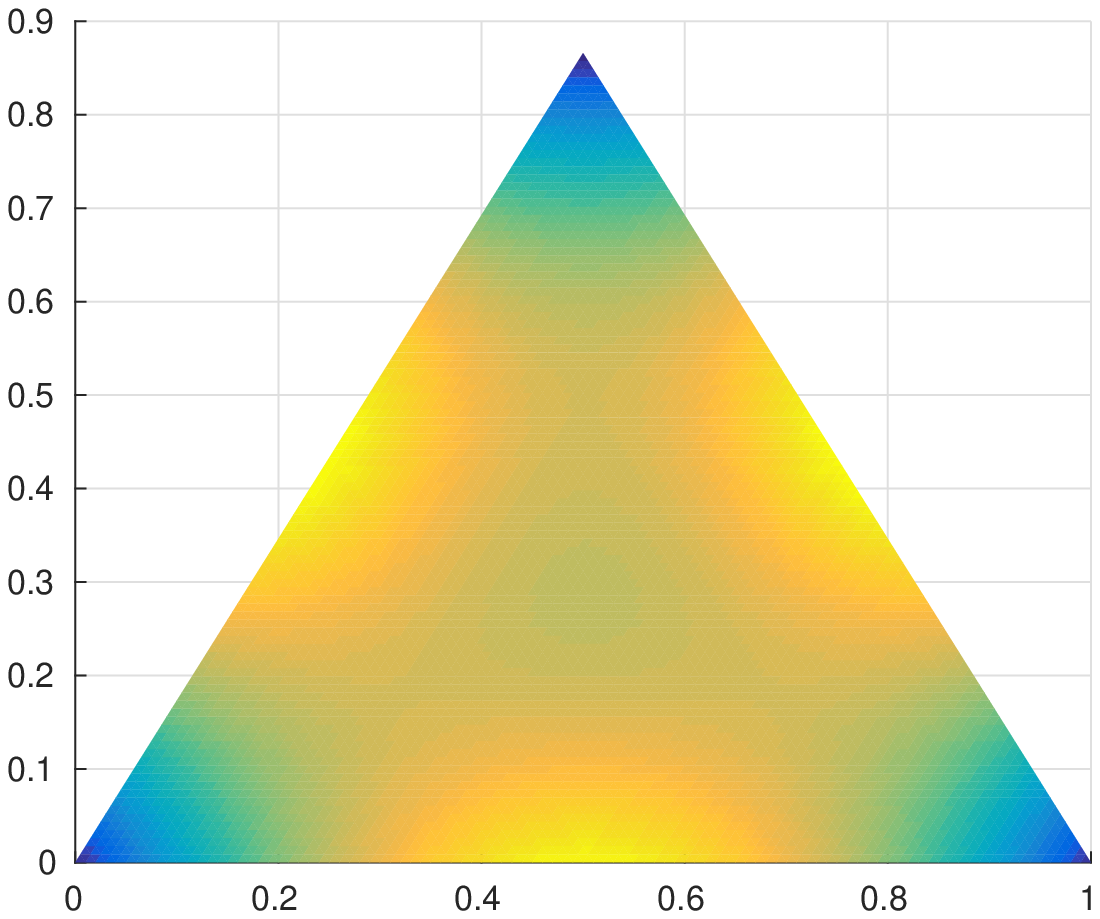}
\end{center}
\caption{\label{fig:regular-eigfunc} The first eigenfunction on the regular triangle}
\end{figure}

%

\begin{table}[ht]
\begin{center}
\caption{\label{table:constant} Verified bound for interpolation constants over different triangles}

\vskip 0.1cm

\begin{tabular}{|l|c|c|c|c|c|}
\hline
$(a,b)$             & $\lambda_{\text{low}}$ & $\lambda_{\text{upper}}$ & $C_{\text{low}}$ & $C_{\text{upper}}$  \\ \hline
$(0,1)$             &          14.8181         &          15.1101           &       0.2571     &     0.2598    \\ \hline  
$(0,\sqrt{3}/3)$    &          21.4906         &          22.1234           &       0.2125     &     0.2158 \\ \hline 
$(1/2,\sqrt{3}/2)$  &          31.6764         &          32.2821           &       0.1759     &     0.1777 \\ \hline 
$(-1/2,\sqrt{3}/2)$ &          5.15806         &          5.26263           &       0.4358     &     0.4404 \\ \hline 
\end{tabular}
\end{center}
\end{table}

\begin{figure}[ht]
\begin{center}
\includegraphics[width=3.2in,angle=0]{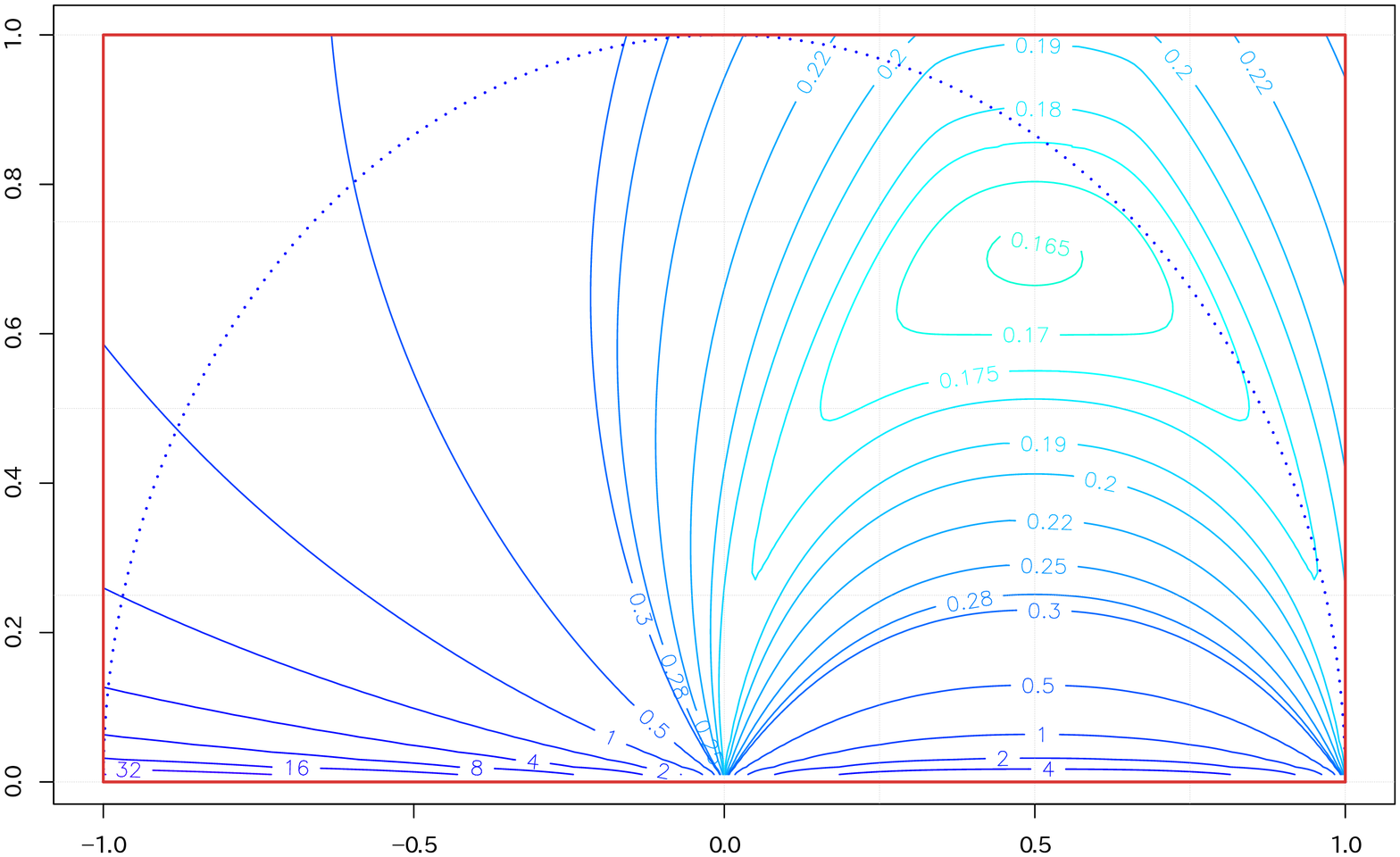}
\caption{\label{fig:P2-C-contour} Contour line of $C_T$ with respect to $(a,b)$ in rectangular $[-1, 1]\times [0.01, 1]$ }
\end{center}
\end{figure}

To investigate the dependency of $C_T$ on the shape of triangle $T$, in \Cref{fig:P2-C-contour},
we draw the contour line for $C_T$ respect to the parameter $(a,b)$.
For this purpose, we choose grid value of $(a,b)$  in $[-1, 1] \times [0.01, 1]$ and evaluate $C_T$ by using finite element method.
As we can see from the numerical results, for triangle with maximum angle near to $\pi$, the value of $C_T$ will be large.
Thus, in finite element computing, such kind of elements are not recommended.

An interesting thing about the numerical results is that the approximate eigenvalue $\lambda_{1,h}$ itself works as the
lower bound for $\lambda_1$, which is due to the usages of the nonconforming FEM.
As shown by asymptotic analysis, if the mesh size is small enough, the approximate eigenvalues from many non-conforming FEMs
give lower bounds for the exact eigenvalues; see, e.g., \cite{yang2010eigenvalue,Luo2012}. However, it is usually difficult to verify
when the mesh is ``small enough".

\section{Summary and future work}\label{sec:summary}
In this paper, we provide explicit bound for the interpolation error constant appearing in the second order Lagrange interpolation function for 
$u\in H^{2}(T)$.

The method proposed here can be applied to bounding error constants of general interpolation error estimation. 
For example, to estimate the constant $C_{T,i}$ in \cref{eq:error-est-pi2-known}, 
instead of considering $u$ directly, let us introduce $v=u_{x}$ and $w=u_{y}$ along with the constraint
condition $v_{y}=w_{x}$. The Fujino-Morely FEM can be used to approximate $v$ and $w$. 
For $i=1,2$ in \cref{eq:error-est-pi2-known}, by solving the corresponding eigenvalue problem for $v$ and $w$, 
it is possible to have optimal lower bound for the first eigenvalue and hence the upper bound of $C_{T,i}$. 
In case of $i=0$, a rough bound of $C_{T,0}$ can be obtained 
by considering the estimations $\|u-\Pi_{2}u\|_{0} \le \widehat{C} |u-\Pi_{2}u|_{1}$ and $|u-\Pi_{2}u|_{1} \le C_{T,1} |u|_{3}$, where 
$\widehat{C}$ is a constant to be evaluated by solving an eigenvalue problem of the Laplacian. \\

As an objective, in the near future,  we plan to create a table with bounds for various interpolation error constants.

\section*{Acknowledgement}

This research is supported by SPS KAKENHI（Grants-in-Aid for Scientific Research) Grant Number 26800090, 16H03950 and 26400194 from Japan
Society for the Promotion of Science (JSPS).

\bibliographystyle{elsarticle-num}
\bibliography{library}


\end{document}